\documentclass[11pt,letterpaper,reqno]{amsart}

\usepackage[pdftex,colorlinks,linkcolor=blue,citecolor=red,filecolor=black]{hyperref}

\usepackage{amsmath}
\usepackage{amsthm}
\usepackage{amssymb}
\usepackage{xfrac}
\usepackage{faktor}
\usepackage{mathtools}
\usepackage{mathrsfs}
\usepackage{tikz}
\usepackage{framed}
\usepackage{upgreek}
\usepackage{bbm}
\usepackage{bm}

\usepackage{enumitem}
\setlist[enumerate]{label=({\roman*})}

\newtheorem{theorem}{Theorem}
\newtheorem{proposition}[theorem]{Proposition}

\newtheorem{lemma}[theorem]{Lemma}

\theoremstyle{definition}

\newtheorem{definition}[theorem]{Definition}

\newtheorem{remark}[theorem]{Remark}

\newtheorem{question}[theorem]{Question}

\numberwithin{equation}{section}
\numberwithin{theorem}{section}

\newcommand\numberthis{\addtocounter{equation}{1}\tag{\theequation}}

\def\<{\langle}
\def\>{\rangle}

\def\Z{\mathbb{Z}}
\def\R{\mathbb{R}}

\def\whM{\widehat M}

\def\whX{\widehat X}

\def\wtV{\widetilde V}
\def\Gab{\Gamma^{\mathrm{ab}}}
\begin{document}

\title[Transitivity and an abelian Liv\v{s}ic theorem]{
Transitivity and an abelian Liv\v{s}ic theorem for covers}
\author{Mark Pollicott and Richard Sharp}
\address{Mark Pollicott, Mathematics Institute, University of Warwick,
Coventry CV4 7AL, UK}
\email{masdbl@warwick.ac.uk }
\address{Richard Sharp, Mathematics Institute, University of Warwick,
Coventry CV4 7AL, UK}
\email{R.J.Sharp@warwick.ac.uk}
\thanks{\copyright 2025. This work is licensed by a CC BY license.}
\begin{abstract}
We show that the abelian Liv\v{s}ic theorem recently obtained by
 A. Gogolev and F. Rodriguez Hertz for null-homologous periodic orbits of 
homologically full Anosov flows continues to hold 
when restricted to periodic orbits which are trivial
with respect to any regular cover for which the lifted flow is transitive.
\end{abstract}
\maketitle
\section{Introduction}
The classical  theorem by  A. N. Liv\v{s}ic  from 1972 states  that  when a real valued H\"older 
continuous cocycles vanishes around the periodic orbits of a hyperbolic system  then it is necessarily a coboundary \cite{Li}.
This is a simple, but elegant, illustration of a dynamical setting  where local data can give global properties.

More recently, there has been interest in variants  of the Liv\v{s}ic  theorem 
 for real valued cocycles which  vanish around a smaller collection of periodic orbits.
 For example,
A. Gogolev and F. Rodriguez Hertz
considered the  case of homologically full Anosov flows and  real valued H\"older cocycles 
which vanish around those periodic orbits which  are null-homologous. Under this hypothesis, they showed
that the cocycle is the sum of a coboundary and a closed $1$-form evaluated on the vector field generating the flow. They  dubbed their result an
 {\it abelian Liv\v{s}ic theorem} (owing to the presence of an additional ``abelian'' term given by a closed $1$-form).
 We note that the homologically full condition is equivalent to transitivity of the lifted flow on the universal homology
 cover.
 Subsequently, the second author showed that the same conclusion holds if the cocycle vanishes around periodic orbits which have trivial Frobenius class with respect to some amenable cover on which the lifted flow is transitive
 \cite{Sharp2024}.
 In this paper, we show that the hypothesis that the lifted flow is transitive is, in fact, sufficient for the conclusion of the abelian Liv\v{s}ic theorem to hold, without any other condition on the covering group.

 Let us formulate our result more precisely. Let $X^t : M \to M$ be a transitive Anosov flow 
 generated by a vector field $X$ and let $L_X$ denote the corresponding Lie derivative.
 Let $\mathcal P$ denote the set of prime periodic orbits for $X^t$.
 Given a continuous function $f : M \to \R$ and $\wp \in \mathcal P$, with least period $\ell(\wp)$, we write
 \[
 \int_\wp f = \int_0^{\ell(\wp)} f(X^t(x)) \, dt
 \]
 for any $x$ on $\wp$. Let $\widehat M$ be a regular cover of $M$ with covering group $\Gamma$
 (with identity element $e_\Gamma$),
 and let $\whX^t : \whM \to \whM$ be the lifted flow. Given a point $x$ on $\wp$ and a choice of lift
 $\hat x$ to $\whM$, we have $\whX^{\ell(\wp)}\hat x = \hat x\gamma$, for some $\gamma \in \Gamma$.
 The conjugacy class of $\gamma$ is independent of the choices of $x$ and $\hat x$: we call it the Frobenius class of $\wp$ and denote it by $\langle \wp \rangle_\Gamma$.
 Let 
 \[
 \mathcal P(\Gamma) = \{\wp \in \mathcal P \hbox{ : } \langle \wp \rangle_\Gamma = \{e_\Gamma\}\},
 \]
 the set of prime periodic orbits with trivial Frobenius class (i.e. which lift to periodic orbits on $\whM$).
Our main result is now the following.

\begin{theorem}\label{th:intro_main_anosov}
Let 
$X^t : M \to M$ be a transitive Anosov flow and let
$\whM$ be a regular cover of $M$, with covering group $\Gamma$, such that
the lifted flow $\whX^t : \whM \to \whM$ is transitive.
If a H\"older continuous function $f : M \to \mathbb R$ satisfies 
\begin{equation}\label{null_periodic}
\int_\wp f = 0  \quad \forall \wp \in \mathcal P(\Gamma)
\end{equation}
then
\[
f =  L_Xu + \omega(X),
\]
for some  H\"older continuous $u : M \to \mathbb R$
which is continuously differentiable along flow lines
and some smooth closed $1$-form $\omega$ on $M$.
\end{theorem}

 If $f$ and $X$ are  $C^\infty$ (and thus $f - \omega(X) = L_Xu$ is also $C^\infty$) we can apply a result of de la Llave,
 Marco and Moryian (Theorem 2.1 of \cite{LMM}) to deduce that $u$ is $C^\infty$.

The question of whether the lift of an Anosov flow on a cover is transitive is interesting in its own right. 
The general situation is not well understood but a recent paper by
Barthelm\'e and Lu \cite{BL} has some results for
Anosov flows on $3$-manifolds. However, when $X^t$ is a geodesic flow over a compact manifold $V$ with negative sectional 
curvatures, the classic results of Koebe, L\"obell and Morse for surfaces \cite{Hedlund-BAMS1939} and Eberlein \cite{Eberlein} in higher dimensions, tell us
that transitivity holds for
the geodesic flow over any regular cover of $V$, apart from over the universal cover $\wtV$.

Let us outline the contents of the paper.
In section \ref{sec:subshifts}, we shall first consider the simpler setting of subshifts of finite type and skew products.  
This is both of independent interest and also illustrative of some of the ideas which will be useful later for Anosov flows. 
Then, in section \ref{sec:Anosov}, we shall consider the case of Anosov flows and 
regular covers.  
In section \ref{sec:geodesic}, we consider an application to marked length spectrum rigidity.
Finally, in section \ref{sec:lie_groups}, we will also consider the 
 case of cocycles valued in other connected Lie groups $G$.

\section{Subshifts of finite type}\label{sec:subshifts}

We first  introduce the  ingredients in the statement for subshifts of finite type.
Let $A$ be a $k \times k$ irreducible  matrix $A$ ($k \geq 2$) with entries in $\{0,1\}$. Writing $A(i,j)$ for the $(i,j)$-entry
of $A$, we say that $A$ is \emph{irreducible} if, for each pair of indices $i,j$, there exists $n \ge 1$ such that $A^n(i,j) >0$.

\begin{definition}
Given $A$ as above, we set 
$$
\Sigma_A = \left\{x = (x_n)_{-\infty}^\infty \in \{1, \cdots, k\}^\mathbb Z \hbox{ : } A(x_{n}, x_{n+1}) = 1 \hbox{ for } n \in \mathbb Z\right\}
$$
and let $\sigma: \Sigma_A \to \Sigma_A$ be 
\emph{the subshift of finite type}
defined by $(\sigma x)_n = x_{n+1}$ for $n \in \mathbb Z$.
\end{definition}
The shift map is a homeomorphism with respect to the  associated metric on $\Sigma_A$ defined by 
$d(x,y) = 2^{-N(x,y)}$, where
$$
N(x,y)= 
\sup\left\{n \in \mathbb Z^+ \hbox{ : }  x_i = y_i \hbox{ for } |i| < n\right\}
$$
for $x = (x_n)_{-\infty}^\infty,  y = (y_n)_{-\infty}^\infty \in \Sigma_A.$
The assumption that $A$ is irreducible is equivalent to the map $\sigma: \Sigma_A \to \Sigma_A$ being  transitive (i.e. that there exists a dense orbit).
We note that here, and for the other dynamical systems we consider in the paper, the phase space has no isolated points and so transitivity is equivalent to forward transitivity
(i.e. that there exists a dense forward orbit). We will use the latter condition when we apply transitivity.

Let $\Gamma$ be a countable group (with the discrete topology) with identity element $e_\Gamma$.
Let $\psi: \Sigma_A \to \Gamma$ be a continuous function on the shift space  (which by recoding we can assume without loss of generality only depends on the zeroth  term $x_0$ of the sequence $x = (x_n)_{n=-\infty}^\infty\in \Sigma_A$).
We define a metric on $\Sigma_A \times \Gamma$ by 
$$
\widehat d\left((x,\gamma), (x', \gamma')\right)
= 
\begin{cases}
d(x,x') &\hbox{ if }\gamma = \gamma' \cr
1  &\hbox{ if } \gamma \neq  \gamma'.
\end{cases}
$$

\begin{definition}
We define  the \emph{skew product} 
$T_\psi : \Sigma_A \times \Gamma \to \Sigma_A \times \Gamma$ 
by
\[
T_\psi (x, \gamma) = (\sigma x,    \psi(x) \gamma).
\]
\end{definition}

We will suppose that the skew product  
$T_\psi :\Sigma_A \times \Gamma \to \Sigma_A \times \Gamma$ 
is transitive. We note that this implies that $\Gamma$ is finitely generated.

Given $n \ge 1$, write
\[
\psi_n(x) = \psi(\sigma^{n-1}x) \cdots \psi(\sigma x)  \psi(x).
\]
We then see that $T_\psi^n(x,\gamma)=(\sigma^nx, \psi_n(x) \gamma)$
and hence that $T_\psi^n(x,\gamma) =(x,\gamma)$ for any (and, equivalently, all) 
$\gamma \in \Gamma$ if and only if $\psi_n(x)=e_\Gamma$.

Let $f : \Sigma_A \to \R$ be a H\"older continuous function and, for $n \ge 1$, write
\[
f^n(x) = f(x) +f(\sigma x) + \cdots + f(\sigma^{n-1}x).
\]
If $f^n(x) =0$ whenever $\sigma^nx=x$ then the classical Liv\v{s}ic theorem tells us 
that there exists a H\"older continuous function
$u : \Sigma_A \to \R$ such that $f = u-u\circ \sigma$ \cite{Li}.
The next theorem addresses what happens if $f^n(x)=0$ on the smaller set of periodic points for which $\psi_n(x)=e_\Gamma$.

 \begin{theorem}\label{thmshiftreal}
Let $\sigma: \Sigma_A \to \Sigma_A$ be a transitive subshift of finite type, $\Gamma$ 
a countable group and $\psi : \Sigma_A \to \Gamma$ a function satisfying 
$\psi((x_n)_{n=-\infty}^\infty)=\psi(x_0)$.
Suppose that $T_\psi : \Sigma_A \times \Gamma \to \Sigma_A \times \Gamma$
is transitive.
If
$f: \Sigma_A \to \mathbb R$ is a H\"older continuous function such that
\[
f^n(x) =0 \mbox{ whenever } \psi_n(x) =e_\Gamma
\]
then there exist
a H\"older continuous function  $u: \Sigma_A \to \R$ 
and a homomorphism $\alpha: \Gamma \to \mathbb R$
such that 
$$
f(x) = u(\sigma x)  - u(x) + \alpha(\psi(x)).
$$
\end{theorem}
 \medskip

We  now give the  proof of this theorem. 

\medskip
\noindent
{\bf Step 1.}
We can define $\widehat f: \Sigma_A \times \Gamma \to \mathbb R$ to be the trivial lift given by   $\widehat f(x, \gamma) = f(x)$  for 
$(x, \gamma) \in \Sigma_A \times \Gamma$.  Then following the classical proof of the Liv\v{s}ic  theorem \cite{Li}
 we can use the transitivity of $T_\psi$ to choose a point $\widehat z = (z, \zeta) \in \Sigma_A \times \Gamma$
(where  $z = (z_k)_{k=-\infty}^\infty$)
 whose 
 $T_\psi$-forward orbit 
\[
\mathcal O(\widehat z) := \bigcup_{n=0}^\infty T_\psi^n(\widehat z)
\]
is dense in $\Sigma_A \times \Gamma$. 
We can then  define $\widehat u : \mathcal O(\widehat z) \to \mathbb R$ by
$$
\widehat u(
T_\psi^n(\widehat z)
)
= 
\sum_{k=0}^{n-1} \widehat  f(T_\psi^k(\widehat z)), \hbox{ for } n \geq 1.
$$
We  will use  the following simple estimate.

\begin{lemma}\label{anosov} There exists $C>0$ such that 
if $n, m > 0$ satisfy 
$$
\widehat d(
T_\psi^n (\widehat z),
T_\psi^{n+m} (\widehat z))
 < 1
$$
then
$$
| \widehat u(T_\psi^n (\widehat z)) - 
\widehat u (T_\psi^{n+m} (\widehat z))|
\leq C \widehat d(
T_\psi^n (\widehat z),
T_\psi^{n+m} (\widehat z)
)^\theta,
$$
where $\theta$ is the H\"older exponent of $f$.
\end{lemma}

\begin{proof}
We can modify the  standard argument  from \cite{Li}
to apply to 
$T_\psi: \Sigma_A \times \Gamma \to \Sigma_A \times \Gamma$.  For completeness, we include some details. 
By the definition of $\widehat d$, we can choose the largest $N$ with 
$\widehat z_{n+i} = \widehat z_{n+m + i}$
for $|i| < N$.  
We can then construct a $T_\psi$-periodic point 
$\widehat w = T_\psi^m (\widehat w) \in \Sigma_A \times \Gamma$ by infinitely concatenating the
following string of symbols
$$(z_{n}, \zeta_n), (z_{n+1}, \zeta_{n+1}), \cdots,  (z_{n+m-1}, \zeta_{n+m-1}),$$
where  $T_\psi^k(\widehat z) = (\sigma^k z, \zeta_k)$ for $k \geq 0$.
In particular, 
$$\widehat d (T_\psi^{n+ i} (\widehat z), T_\psi^i (\widehat w)) \leq 2^{-N - \min\{i, m-i\}}
\hbox{ for } 0 \leq i \leq m-1.
$$
 
By definition, we can write
$$
\widehat u(T_\psi^{n+m} (\widehat z)) - 
\widehat u (T_\psi^{n} (\widehat z))
= \sum_{k=n}^{n+m-1} \widehat  f(T_\psi^k(\widehat z) ).
$$
Then by the triangle inequality we can bound
\begin{align*}
&\left|\sum_{k=n}^{m+n-1} \widehat  f(T_\psi^k(\widehat z) ) \right|
= \left|
\sum_{k=n}^{n+ m-1} \widehat  f(T_\psi^k(\widehat z) )
- 
 \sum_{k=0}^{m-1} \widehat  f(T_\psi^k(\widehat w))
\right| \\
&\leq 
\sum_{k=0}^{m-1}
|  f(\sigma^{n+k}z )
- 
 f(\sigma^kw )
| \\
&=\sum_{k=0}^{[m/2]}
|  f(\sigma^{n+k}z )
- 
 f(\sigma^kw )
| 
+
\sum_{k=[m/2] +1}^{n-1}
|  f(\sigma^{n+k}z )
- 
 f(\sigma^kw )
| \\
&\leq 
2 \left( 
\frac{\|f\|_\theta }{1 - 2^{-\theta}}
\right) 2^{-\theta N}.
\end{align*}
The conclusion then follows.
\end{proof}

We can deduce from Lemma \ref{anosov}  that $\widehat u$ is locally H\"older continuous 
 on the dense  set 
 $\mathcal O(\widehat z)$ 
and therefore extends to a locally H\"older continuous function $\widehat u:\Sigma_A \times \Gamma \to \mathbb R$ satisfying 
\begin{equation}\label{eq:u_extends_to_holder}
\widehat u (T_\psi (x, \gamma))
- \widehat u  (x, \gamma) = \widehat f(x, \gamma) 
\end{equation}
for all $x\in \Sigma_A$ and $\gamma \in \Gamma$.

\bigskip
\noindent
{\bf Step 2.} 
Let $y \in \Sigma_A$ and $\gamma_1, \gamma_2
 \in \Gamma$.
Given any $0<\delta <1$, we can 
use the density of 
the orbit of $\widehat z =(z,\zeta)$
to choose $n, m > 0$ such that 
$$
d(\sigma^nz, y ) < \delta, \quad d(\sigma^{n+m}z, y ) < \delta,
$$
$
 \psi_n(z)\zeta = \gamma_1
$
and $ \psi_{n+m}(z) \zeta =\gamma_2$.
In particular, 
$$
\max\left\{
|\widehat u (y, \gamma_1) - \widehat  u(T_\psi^n(\widehat z))|,
|\widehat u (y, \gamma_2) - \widehat u (T_\psi^{n+m}(\widehat z)|
\right\} < \eta, 
$$
where 
\begin{align*}
\eta &= 
\sup_{\widehat d((x,\gamma),(x',\gamma'))<1}\left\{
\frac{|\widehat u(x,\gamma)-\widehat u(x',\gamma')|}
{\widehat d((x,\gamma),(x',\gamma'))^\theta}\right\}
\delta^\theta
\\
&=
\sup_{\substack{d(x,x')<1 \\\gamma \in \Gamma}}
\left\{
\frac{|\widehat u(x,\gamma) - \widehat u(x',\gamma)|}{d(x,x')^\theta}
\right\} \delta^\theta.
\end{align*}
Therefore, by the triangle inequality,
$$
\left|\widehat u (y, \gamma_1) - \widehat u (y, \gamma_2)
- 
\sum_{k=n}^{n+m-1} f(\sigma^k z)
\right| < 2\eta.
$$
Since $\delta>0$ 
can be chosen arbitrarily small and $\sum_{k=n}^{n+m-1} f(\sigma^k z)$ is evidently  independent of $\Gamma$, we can deduce that 
$\widehat u (y, \gamma_1) - \widehat u (y, \gamma_2)$ is independent of the action of  $\Gamma$.
More precisely, a comparison with the corresponding expression for the forward obit of $(z,\zeta \gamma)$ gives that,
for all
$\gamma_1, \gamma_2, \gamma \in \Gamma$,
\begin{equation}\label{eq:independent_of_Gamma}
\widehat u (y,  \gamma_1\gamma ) - \widehat u (y,\gamma_2 \gamma)
=
\widehat u (y, \gamma_1) - \widehat u (y, \gamma_2). 
\end{equation}

\bigskip
\noindent
{\bf Step 3.} 
By construction, we can write for any $x \in \Sigma_A$ and $\gamma_1, \gamma_2 \in  \Gamma$:
\begin{align*}
\widehat  u(T_\psi (x, \gamma_1)) -  \widehat  u(x, \gamma_1) 
&= \widehat f (x, \gamma_1)
  = f(x) 
=  \widehat f(x, \gamma_2) \\
&= \widehat  u(T_\psi (x, \gamma_2)) - \widehat  u(x, \gamma_2).
\numberthis \label{eq:T_invariant}
\end{align*} 
Rearranging (\ref{eq:T_invariant}) gives 
\begin{equation}\label{eq:T_invariant_two}
\widehat  u(T_\psi (x, \gamma_1)) - \widehat  u(T_\psi (x, \gamma_2))
=\widehat  u(x, \gamma_1) -   \widehat u(x, \gamma_2).
\end{equation}
Now choose $\gamma \in \Gamma$ and set 
\[
\Delta_\gamma(x,\gamma_1) = \widehat u(x,\gamma_1\gamma) - \widehat u(x,\gamma_1).
\]
Applying (\ref{eq:T_invariant_two}) with $\gamma_2 = \gamma_1\gamma$
gives that
\[
\Delta_\gamma(T_\psi(x,\gamma_1))=\Delta_\gamma(x,\gamma_1),
\]
for all $x \in \Sigma_A$ and all $\gamma_1 \in \Gamma$,
i.e. $\Delta_\gamma : \Sigma_A \times \Gamma \to \R$ is $T_\psi$-invariant.  
Since $T_\psi$ is transitive we can deduce that $\Delta_\gamma$ takes a constant value
$\alpha (\gamma)$.
In particular, setting $\gamma_1 = e_\Gamma$ gives
\[
\widehat u(x,\gamma)= \widehat u(x,e_\Gamma) +\alpha(\gamma)
\]  
and we define a function $u : \Sigma_A \to \R$ by $u(x)=\widehat u(x,e_\Gamma)$.
Then the identity (\ref{eq:u_extends_to_holder})
gives that 
\begin{align*}
f(x) = f(x,e_\Gamma) &=
\widehat u(T_\psi(x,e_\Gamma)) - \widehat u(x,e_\Gamma)
\\
&=\widehat u(\sigma x,\psi(x)) -\widehat u(x,e_\Gamma)
\\
&= \widehat u(\sigma x,e_\Gamma)  - \widehat u(x,e_\Gamma) +\alpha(\psi(x))
\\
&= u(\sigma x) -u(x) +\alpha(\psi(x)).
\numberthis \label{eq:formula_for_f}
\end{align*}

\bigskip
\noindent
{\bf Step 4.} We claim that $\alpha: \Gamma \to \mathbb R$ is a homomorphism.  
By definition, $\alpha(e_\Gamma)=0$.
Given $\gamma_1, \gamma_2 \in \Gamma$ we can write 
for any $x\in \Sigma$:
\begin{align*}
\alpha(\gamma_1\gamma_2) &=
   \widehat u(x, \gamma_1\gamma_2)- \widehat  u(x,  e_\Gamma) \\
&= \left(\widehat  u(x,  \gamma_1\gamma_2) -   \widehat u(x, \gamma_2)\right)
+ \left( \widehat  u(x,  \gamma_2) -   \widehat u(x, e_\Gamma)\right)
\\
&= \left(\widehat  u(x,  \gamma_1\gamma_2) -   \widehat u(x, \gamma_2)\right) + \left( \widehat  u(x, \gamma_2) -  \widehat u(x, e_\Gamma) \right)\cr
&=\alpha(\gamma_1) + \alpha(\gamma_2).
\end{align*}
This completes the proof of Theorem \ref{thmshiftreal}.

\section{Anosov flows} \label{sec:Anosov}

We will now present the proof of Theorem \ref{th:intro_main_anosov} using the basic strategy of the proof of Theorem \ref{thmshiftreal} in the previous section. We begin with some standard notation and definitions.

Let $X^t : M \to M$ ($t\in \mathbb R$) be  a $C^1$ 
flow on a smooth compact Riemannian manifold 
 $M$ and let $X$ denote the associated vector field. 

 \begin{definition}
 We say that $X^t : M \to M$ is  an \emph{Anosov flow}
if there is a continuous $DX^t$-invariant splitting $TM = E^0 \oplus E^s \oplus E^u$ such that
\begin{enumerate}
\item[(1)] $E^0$ is one-dimensional and tangent to the flow direction; and 
\item[(2)] there exist constants $C, \lambda > 0$ such that
\begin{enumerate}
\item $\|DX^t| E^s\| \leq C e^{-\lambda t}$ for $t \geq 0$; and 
\item $\|DX^{-t}| E^{u}\| \leq C e^{-\lambda t}$ for $t \geq 0$.
\end{enumerate}
\end{enumerate}
\end{definition}

We shall also assume that $X^t: M \to M$ is  transitive.

Let $\widehat M$ be a regular cover for $M$ (with the lifted metric) and let $\pi: \widehat M \to M$ 
denote
 the associated projection.
We write $\Gamma \cong \pi_1(M)/\pi_1(\widehat M)$ for the covering group, which acts freely (on
the right) on $\widehat{M}$.

We let  $\widehat{X}$ and
$\widehat X^t :  \widehat M \to  \widehat M$ denote, respectively, the lifts of  $X$ and $X^t: M \to M$
to
$\widehat{M}$.
Then $\pi \circ \widehat X^t = X^t \circ \pi$, for all $t\in \mathbb R$, and
$\widehat{X}^t(x\gamma) = \widehat{X}^t(x)\gamma$, for all $x\in \widehat{M}$, $t \in \R$ and $\gamma \in \Gamma$.

As in the introduction we assume that $\widehat{X}^t :\widehat{M} \to \widehat{M}$ is transitive.

\begin{remark}[Coudene, \cite{coudene} Theorem 3]
The following gives an equivalent characterization of transitivity of $\widehat X^t: \widehat M \to \widehat M$.
The $\widehat \phi$ recurrent points 
are  dense if and only if $\widehat X^t$ is forward transitive.
In particular, if   the  closed orbits 
are  dense then the flow is forward transitive.
\end{remark}

Let $\Gamma^{\mathrm{ab}} = \Gamma/[\Gamma,\Gamma]$ denote the abelianization of $\Gamma$ and let 
$\phi : \Gamma \to \Gamma^{\mathrm{ab}}$ denote the natural quotient homomorphism.
Clearly, any homomorphism from $\Gamma$ to $\R$ factors through $\Gamma^{\mathrm{ab}}$ so that
$\alpha = \bar{\alpha} \circ \phi$, for some homomorphism $\bar{\alpha} : \Gamma^{\mathrm{ab}} \to \R$,
and $\alpha \mapsto \bar \alpha$ gives an isomorphism between $\mathrm{Hom}(\Gamma,\R)$ 
(the group of homomorphisms from $\Gamma$ to $\R$) and 
$\mathrm{Hom}(\Gamma^{\mathrm{ab}},\R)$.
Furthermore, $\Gamma^{\mathrm{ab}}$ is a quotient of $H_1({\color{black}M},\Z)$ via a surjective homomorphism 
$\kappa : H_1(M,\Z) \to \Gab$.
We can canonically identify, $\mathrm{Hom}(\Gamma^{\mathrm{ab}},\R)$,
and hence $\mathrm{Hom}(\Gamma,\R)$, with a subgroup of $H^1(M,\R)$.
More precisely, there is a canonical isomorphism
\[
\iota : \mathrm{Hom}(\Gamma,\R) \to
\Lambda :=
 \{w \in H^1(M,\R) \hbox{ : } \langle w, h \rangle =0 \ \forall 
h \in \ker \kappa
\},
\]
where $\langle\cdot,\cdot \rangle$ is the natural pairing between cohomology and homology.
In particular, $\alpha$ determines 
a unique
harmonic $1$-form $\omega$ with cohomology class $[\omega] = \iota(\alpha)$,
which is defined by the condition
\[
\int_{x_0}^{x_0\gamma} \widehat{\omega} = \alpha(\gamma),
\]
where $x_0$ is any point in $\widehat{M}$, $\widehat{\omega}$ is the lift of $\omega$ to
$\widehat{M}$ and the integral is taken over any path from $x_0$ to $x_0\gamma$; 
since $[\omega]\in \Lambda$, the integral is independent of 
these choices.
(Instead of requiring $\omega$ to be harmonic, it suffices the take $\omega$ to be a closed $1$-form with $[\omega] =\iota(\bar{\alpha})$ but then
it is only unique modulo the addition of an exact form.)

If $\widehat{\omega(X)}$ is defined to be the lift $\omega(X) \circ \pi$ then we observe that
\[
\widehat{\omega(X)} = \widehat{\omega}(\widehat{X}).
\]
We can now replace $ f : M \to \mathbb R$ by  $f_{\omega}: M \to \mathbb R$ defined by 
$$
f_\omega:= f - \omega(X),
$$
where $\omega$ is an arbitrary harmonic $1$-form with $[\omega] \in \Lambda$.
Since $f$ integrates to zero around every periodic orbit in $\mathcal P(\Gamma)$, so does $f_\omega$.

Following these preliminaries we now turn to the proof of the theorem.

\medskip
\noindent
{\it Proof of Theorem \ref{th:intro_main_anosov}.}
We divide the proof 
into simpler steps.

\bigskip
\noindent
{\bf Step 1.}
 Let $\widehat f_\omega  : \widehat M \to \mathbb R$ be the
lift to $M$  of $ f_\omega : M \to \mathbb R$ {\color{black} defined by} $\widehat f  = f  \circ \pi$.
 Following the classical proof of the Liv\v{s}ic  theorem we can use
 the transitivity of $\widehat{X}^t$ to choose a point $z \in \widehat M$ such that 
the forward orbit
$$\mathcal O(z) := \bigcup_{t=0}^\infty \widehat X^t (z)$$
is dense in $\widehat M$. 
We then define $\widehat u_\omega : \mathcal O(z) \to \mathbb R$ by
$$
\widehat u_\omega (
\widehat X^t (z)
)
= 
\int _{0}^t  \widehat  f_\omega (\widehat X^v (z)) \, dv , \hbox{ for } t > 0.
$$

We  will use  the following useful estimate.

\begin{lemma} There exists $C>0$ and $\epsilon > 0$ such that 
if $t, s > 0$ with 
$$
d(
\widehat X^t (z),
\widehat X^{t+s} (z)) < \epsilon
$$
then
$$
| 
 \widehat  u_\omega(\widehat X^t (z))
 -  \widehat  u_\omega (\widehat X^{t+s} (z)) |
\leq C d(
\widehat X^t (z),
\widehat X^{t+s} (z)
)^\theta
$$
\end{lemma}

\begin{proof}
The proof is analogous to the proof of Lemma \ref{anosov} in the discrete case.
A key ingredient is the Anosov closing lemma 
(see Lemma 13.1 in \cite{An})
which guarantees that there exist $C, \epsilon > 0$ such that,  for any $0 < b < \lambda$, 
if, for some $t, s > 0$,
one can bound 
$d(
\widehat X^t (z),
\widehat X^{t+s} (z)
) < \epsilon
$ 
then there exists 
$T > 0$ and
a point $y =\widehat X^T(y)$ 
satisfying $|T- s| \leq C \epsilon$ and 
$$
d(\widehat X^{\tau+t} (z), \widehat X^\tau (y)) \leq C e^{- b \min\{\tau, s-\tau\}}  
d(\widehat X^t (z), \widehat X^{s+t} (z)),
$$
for $0 \le \tau \le s$.

By definition, we can write
  $$
\widehat  u_\omega(\widehat X^t (z))
 -  \widehat  u_\omega (\widehat X^{s+t} (z))
 = \int_0^s \widehat f_\omega (\widehat X^{\tau+t} (z)) \, d\tau.
$$
Furthermore, using the closing lemma  {\color{black} stated above we can bound} 
\begin{align*}
\left|\int_0^s \widehat f_\omega (\widehat X^{t+\tau}(z)) \, d\tau\right|&=
\left| \int_0^s \widehat f_\omega (\widehat X^{t+\tau}(z)) \, d\tau -  
\underbrace{\int_0^T \widehat f_\omega (\widehat X^{\tau}(y))\, d\tau}_{=0}\right| \\
&\leq \int_0^s \|f_\omega\|_\theta d(\widehat X^{t+\tau} (\widehat x_0), \widehat X^\tau(\widehat y))^\theta \, d\tau   + |T-s| \|f\|_\infty
\\
&\leq 
2d(\widehat X^t (z), \widehat X^{s+t} (z))^\theta \|f_\omega\|_\theta 
 \int_0^\infty
e^{-b\theta \tau} \,
 d\tau   + |T-s| \|f\|_\infty
\\
&\leq K  d(\widehat X^t (z), \widehat X^{s+t} (z))^\theta,  
\end{align*}
say, for some $K>0$ which is independent of $d(\widehat X^t (z), \widehat X^{s+t} (z))$.
\end{proof}

This shows that $\widehat u_\omega$ is 
locally
H\"older continuous on the dense forward orbit 
$\mathcal O(z)$
of $z$ and 
thus
extends to a 
locally
H\"older continuous
 function $\widehat u_\omega :
\widehat  M\to \mathbb R$.   Moreover, by construction we see that
$\widehat u_\omega$ is differentiable along flow lines with
\begin{equation}\label{eq:cohom_eq_on_cover}
L_{\widehat{X}}\hat{u}= \widehat f_\omega  = f _\omega\circ \pi. 
\end{equation}

\bigskip
\noindent
{\bf Step 2.} 
Let $x \in \widehat M$, $T\ge 0$ and $\gamma_1 \in \Gamma$.
Given any $0<\delta \le 1$ we can choose $t, s > 0$ such that 
$$
d( \widehat  X^t (z), \widehat X^T(x) \gamma_1 ) < \delta 
\hbox{ and } d( \widehat  X^{s+t} (z), x) < \delta.
$$
In particular, letting
\[
\eta = \sup_{0<d(x,y) \le 1} \frac{|\widehat u_\omega(x)- \widehat u_\omega(y)|}{d(x,y)^\theta} \delta^\theta,
\]
the triangle  inequality gives 
$$
\left|\widehat u_\omega (\widehat X^T(x)\gamma_1) - \widehat u_\omega (x)
- 
\int_0^s f_\omega  (X^{t+\tau}  (z) ) \, d\tau 
\right| < 2\eta.
$$
Since $\delta>0$ is arbitrary  and, for all $\gamma \in \Gamma$,
\[
\int_0^s \widehat f_\omega  (  \widehat  X^{t+\tau}  (z)) ) \, d\tau = 
\int_0^s f (X^{t+\tau}  (\pi(z))) \, d\tau 
=\int_0^s \widehat f_\omega  (  \widehat  X^{t+\tau}  (z\gamma)) ) \, d\tau,
\]
we can deduce that 
${\color{black} x} \mapsto \widehat u_\omega (\widehat X^T(x)  \gamma_1) - \widehat u_\omega (x )$ is $\Gamma$-invariant,
i.e. 
for all $x \in \widehat{M}$, $T \ge 0$ and 
$\gamma_1, \gamma \in \Gamma$,
\begin{equation}\label{eq:uhat_diff_inv}
\widehat u_\omega (\widehat X^T(x)  \gamma\gamma_1 ) - \widehat u_\omega (x \gamma)
= \widehat u_\omega (\widehat X^T(x)  \gamma_1) - \widehat u_\omega (x).
\end{equation}

\bigskip
\noindent
{\bf Step 3.} 
Rearranging (\ref{eq:uhat_diff_inv}) gives 
\begin{equation} \label{eq:uhat_diff_inv_rearranged}
\widehat u_\omega ( \widehat X^T(x) \gamma\gamma_1 )
-  \widehat u_\omega (\widehat X^T(x)\gamma_1)
 =  \widehat u_\omega (x \gamma) 
- \widehat u_\omega (x )
\end{equation}
for any  $T \ge 0$.
Fixing $\gamma \in \Gamma$, we define $\Delta_\gamma : \widehat{M} \to \R$ by 
$$
\Delta_\gamma(x) : = \widehat u_\omega(x\gamma ) - \widehat u_\omega(x ).
$$
Applying (\ref{eq:uhat_diff_inv_rearranged}) with $ \gamma_1 =e_\Gamma$ gives that 
for all $x \in \widehat M$ we have  $\Delta_\gamma : \widehat M \to \mathbb R$ is $\widehat X^t$-invariant, i.e. 
$\Delta_\gamma(\widehat X^t(x)) = \Delta_\gamma(x)$, for all $t \ge 0$.
Since  $\widehat X^t$ is transitive we can deduce that $\Delta_\gamma$ takes a constant value $\alpha_\omega(\gamma)$, say.
In particular, 
\begin{equation}\label{eq:alpha_is_obstruction}
 \widehat u_\omega (x\gamma ) =  \widehat u_\omega (x) + \alpha_\omega(\gamma). 
\end{equation}

\bigskip
\noindent
{\bf Step 4.} We claim that the map $\alpha_\omega: \Gamma \to \mathbb R$ is a homomorphism.  Given $\gamma, \gamma' \in \Gamma$ we can write 
for any $x\in \widehat{M}$:
\begin{align*}
\alpha_\omega(\gamma\gamma') &=
 \widehat u_\omega (x \gamma \gamma') -  \widehat u_\omega (x)   
 \\
&= \left( \widehat u_\omega (x  \gamma  \gamma') - \widehat u_\omega (x  \gamma) \right)
{\color{black} + } \left(\widehat u_\omega (x  \gamma)  - \widehat u_\omega (x) \right)
\\
&= \left( \widehat u_\omega (x  \gamma') -  \widehat u_\omega ( x)  \right)  
{\color{black} + }  \left( \widehat u_\omega ({\color{black} x} \gamma)  - \widehat u_\omega (\color{black} x) )\right)
\\
&=\alpha_\omega(\gamma') + \alpha_\omega(\gamma),
\end{align*}
where for the last step we have used (\ref{eq:uhat_diff_inv}) with $T=0$ and $\gamma_1=\gamma'$.

\medskip
\noindent
{\bf Step 5.} 
 Unlike the proof for skew products there is an extra step required to push a solution down from $\widehat M$
to $M$. To do this, given equation (\ref{eq:alpha_is_obstruction}), we need to choose $\omega$ so that
$\alpha_\omega=0$.
We claim that this holds if we set $\omega$ to be the unique harmonic $1$-form satisfying
$[\omega]= \iota(\alpha_0)$.

For all $y \in \widehat{M}$, we have
\begin{align*}
(\widehat u_\omega-\widehat u_0)(\widehat X^t (y)) -(\widehat u_\omega-\widehat u_0)(y) 
&=- \int_0^t \widehat{\omega(X)}(\widehat{X}^s(y)) \, ds
\\
&=- \int_0^t \omega(X)(X^s(\pi (y))) \, ds.
\end{align*}
Now let $x \in \widehat M$ and $\gamma \in \Gamma$. By transitivity, we can choose
$y \in \widehat{M}$ and $t>0$ such that $y$ is close to $x$ and $\widehat X^{t}(y)$ is close to 
$x\gamma$. In particular, given $\epsilon>0$, we can arrange that
\begin{align*}
&\left|\alpha_\omega(\gamma)-\alpha_0(\gamma)
+ \int_0^t \widehat{\omega(X)}(\widehat{X}^s(y)) \, ds\right| 
\\
&=\left|(\widehat u_\omega-\widehat u_0)(x \gamma) -(\widehat u_\omega-\widehat u_0)(x)  )
+ \int_0^t \widehat{\omega(X)}(\widehat{X}^s(y)) \, ds\right| 
\\
&=|
((\widehat u_\omega-\widehat u_0)(x \gamma) -(\widehat u_\omega-\widehat u_0)(x)  )
-
(\widehat u_\omega-\widehat u_0)(\widehat X^{t} (y) -(\widehat u_\omega-\widehat u_0)(y)) )
|<\epsilon.
\end{align*}
and
\[
\left|\alpha_0(\gamma) -\int_0^t \widehat{\omega(X)}(\widehat X^sy) \, ds\right| =
\left|\int_{x}^{x\gamma} \widehat{\omega} -\int_0^t \widehat{\omega(X)}(\widehat X^sy) \, ds\right| <\epsilon.
\]
Since $\epsilon>0$ can be chosen arbitrarily small we can conclude that, for all $\gamma \in \Gamma$, $\alpha_\omega(\gamma)=0$.
Therefore, $\widehat u_\omega$ descends to $u : M \to \mathbb R$ and,
by (\ref{eq:cohom_eq_on_cover}), we have $f_\omega=L_Xu_\omega$, so that
\[
f = L_Xu +\omega(X).
\]

This completes the proof of Theorem \ref{th:intro_main_anosov}. \qed

\section{Application to marked length spectrum rigidity}\label{sec:geodesic}

Let $S$ be a smooth compact orientable surface of genus at least $2$ and
let $\mathcal F(S)$ denote the set of non-trivial free homotopy classes on $S$.
(This is in natural one-to-one correspondence with the set of non-trivial
conjugacy classes in $\pi_1(S)$.) 
Otal \cite{Otal} and, independently, Croke \cite{Croke} showed that a negatively curved Riemannian metric $g$ on $S$ is determined (up to isometry) by its marked length spectrum.
To state this more precisely, define a map
$\ell_g :\mathcal F(S) \to \R$, where $\ell_g(\alpha)$ is the length of the unique
closed geodesic (with respect to $g$) in $\alpha$. 
Otal and Croke then showed the following.

\begin{proposition}[\cite{Croke}, \cite{Otal}]\label{prop:Otal-Croke}
If $g_1$ and $g_2$ are Riemannian metrics of negative curvature on $S$
such that $\ell_{g_1}=\ell_{g_2}$ then $g_1$ and $g_2$ are isometric. 
\end{proposition}

More generally, let $\widehat S$ be a regular cover of $S$. Given a free homotopy class 
$\alpha$, let $\mathcal F(S,\widehat S)$ denote the set of free homotopy classes on $S$ that lift to closed curves on $\widehat{S}$. 
This set is non-empty unless $\widehat S = \widetilde S$, the universal cover.

The following result was obtained independently by 
Hao (\cite{Hao}, Theorem B and subsequent comments) and
Wan, Xu and Yang (\cite{WXY}, Theorem 1.10). Furthermore, as pointed out in \cite{CR} and \cite{WXY}
this is also implicit in \cite{Bonahon_MSRI}.

\begin{proposition} \label{prop:cover-rigidity}
Let $\widehat S$ be a regular cover of $S$ that is not the universal cover.
If $g_1$ and $g_2$ are Riemannian metrics of negative curvature on $S$
such that $\ell_{g_1}|_{\mathcal F(S,\widehat{S})}=\ell_{g_2}|_{\mathcal F(S,\widehat{S})}$ then $g_1$ and $g_2$ are isometric.
\end{proposition}

In the particular case that 
 $\widehat S$ is the homology cover of $S$, 
Gogolev and Rodriguez Hertz  showed that this result followed from  their abelian Liv\v{s}ic theorem \cite{GRH}.
However, it is interesting to note that using instead the stronger Theorem \ref{th:intro_main_anosov}
recovers  the general result in Proposition \ref{prop:cover-rigidity}.

More precisely, 
let $X^t: M \to M$ be the geodesic flow on the unit tangent bundle $M = T^1S$
of  $S$ with respect to a negatively curved Riemannian metric $g$. This flow is Anosov and transitive \cite{An}.
If $\widehat S$ a regular  cover of $S$ with covering group $\Gamma$, then
 $\widehat M = T^1\widehat S$
  is a regular $\Gamma$-cover for $M$ and the geodesic flow
$\widehat X^t : \widehat M \to \widehat M$ is a $\Gamma$-extension for $X^t : M \to  M$.
Crucially, there is the following useful result.

\begin{proposition}[Hedlund \cite{Hedlund-BAMS1939}]
Let $\widehat S$ be a regular cover for $S$ that is not equal to the universal cover.
Then  $\widehat X^t : \widehat M \to \widehat M$ is transitive.
\end{proposition}

Now suppose that $S$ supports two negatively curved Riemannian metrics $g_1$ and $g_2$, and let $X_1^t : M \to M$
and $X_2^t : M \to M$ be the corresponding geodesic flows (which we can regard as defined on the same space).
Then there is an orbit equivalence $H : M \to M$ that is homotopic to the identity. Furthermore, if $\mathcal P$ denotes the set of prime periodic orbits for $X_1^t$, then the hypothesis $\ell_{g_1}(\alpha) =\ell_{g_2}(\alpha)$ for all
$\alpha \in \mathcal F(S,\widehat S)$ implies that $\wp$ and $H(\wp)$ have the same period for all $\wp \in 
\mathcal P$ with $\langle \wp \rangle_\Gamma = \{e_\Gamma\}$.
Given Theorem \ref{th:intro_main_anosov}, we then
can still apply the arguments of Theorem 5.2 of \cite{GRH} to conclude that 
$X_1^t$ and $X_2^t$ are topologically conjugate. Therefore, $\wp$ and $H(\wp)$ have the same period for \emph{all}
$\wp \in \mathcal P$ and, consequently, $\ell_{g_1}(\alpha) = \ell_{g_2}(\alpha)$ for all $\alpha \in \mathcal F(S)$.
We can then apply Proposition \ref{prop:Otal-Croke} to conclude that $g_1$ and $g_2$ are isometric.

\begin{remark}\label{rem:higher_dim}
The same argument, based on Theorem \ref{th:intro_main_anosov},
works in higher dimensions, in which case transitivity of the geodesic flow lifted to a cover is a theorem of Eberlein (Theorem 3.8 of \cite{Eberlein}). 
Let $V$ be a compact manifold of dimension $n$ with a Riemannian metric $g_0$ of negative sectional curvature.
Then Guillarmou and Lefeuvre \cite{GL} showed that, for $N > 3n/2+8$, there exists $\epsilon>0$ such that
if $g$ is a Riemannian metric on $V$ with the same marked length spectrum as $g_0$ and $\|g-g_0\|_{C^N(V)}
<\epsilon$ then $g$ and $g_0$ are isometric.
Using the above arguments, isometry continues to hold if the marked length spectra agree when restricted to 
free homotopy classes which are trivial with respect to a regular cover $\widehat V$ which is not the universal cover.
\end{remark}

\begin {remark}
In  \cite{BEHLW} there is an analogous  rigidity result to that in \cite{GL}  where the marked length spectrum is replaced by the marked Poincar\'e determinant spectrum.  More precisely, let $V$ be a compact $3$-manifold 
$g$
equipped with a
Riemannian metric of negative sectional curvature and let $X^t : T^1V \to T^1V$ be the corresponding geodesic flow.
Furthermore, let $\Sigma$  be a local Poincar\'e section for 
a periodic orbit $\wp $ for  $X^t$ of length $\ell_g(\wp)$ 
(which is orthogonal to the direction of the flow using  the Sasaki metric).  One  can then associate to $\wp$
the Poincar\'e determinant $p_g(\wp) = \det(DX^{\ell_g(\wp)}|_{E^u})$,  by taking the determinant of the derivative of the Poincar\'e map restricted to the unstable bundle,  and then finally  define  the marked Poincar\'e  determinant spectrum by  $p_g : \mathcal F(V) \to \R$, where $\wp_\alpha$ is the unique periodic orbit corresponding to the free homotopy class $\alpha$ and $p_g(\alpha) := p_g(\wp_\alpha)$.  Then (\cite{BEHLW}, Theorem 1.3) states that, 
if $g$ is a Riemannian metric on $V$ with the  same marked Poincar\'e  determinant spectrum as $g_0$ and $\|g-g_0\|_{C^N(V)}
<\epsilon$ then $g$ and $g_0$ are homothetic.
Using the above arguments, 
the conclusion continues to hold if the marked Poincar\'e determinant  spectra agree when restricted to 
free homotopy classes which are trivial with respect to a regular cover which is not the universal cover.
\end{remark}

\section{Skew products and non-abelian cocycles}\label{sec:lie_groups}

In this section we briefly describe a result analogous to an abelian Liv\v{s}ic theorem for cocycles valued in 
a connected Lie group $G$. 
We will formulate this in the simplest case of skew products.
 
Let $\sigma : \Sigma_A \to \Sigma_A$ be a 
transitive
 subshift of finite type and let $\Gamma$ be
a countable group with the discrete topology.
Let $\psi : \Sigma_A \to \Gamma$ satisfy $\phi((x_n)_{n=-\infty}^\infty)=\phi(x_0)$ and, as in section \ref{sec:subshifts},
define a skew product
$T_\psi : \Sigma_A \times \Gamma \to \Sigma_A \times \Gamma$ by
\[
T_\psi (x,\gamma) = (\sigma x,  \psi(x)\gamma ).
\]
As above, we assume that the skew product is  transitive.

\medskip

Now let $G$ be a connected Lie group with right-invariant metric $d_G$ and suppose 
$f : \Sigma_A \to G$ is an $\theta$-H\"older continuous function.

For cocycles valued in non-abelian groups we require  some additional distortion hypotheses. 
To this end, we define the following quantities.

\begin{definition}
We can quantify the distortion associated to the action of $f$ on $G$ by 
$$
\mu_s = \lim_{n \to +\infty}
\left(
\sup_{x\in \Sigma_A}\|  
\hbox{\rm Ad} (  f_n(x)) 
\|
\right)^{1/n}
\hbox{and }
\mu_u = \lim_{n \to +\infty}
\left(
\sup_{x\in \Sigma_A}
\|  
\hbox{\rm Ad} \left(f_n(x)\right)^{-1}
\|
\right)^{1/n}
$$
where
$f_n(x)= f(\sigma^{n-1}  x) \cdots   f(\sigma x)  f(x)$
and $\mathrm{Ad}(\cdot)$ is the adjoint action on the associated Lie algebra. 
\end{definition}

We impose the following assumption on $f$.

\medskip
\noindent
\emph{Distortion Assumption.} The
H\"older exponent $\theta$ of $f$ is strictly greater than 
\[
\max \left\{
\frac{|\log \mu_s|}{\log 2}, \frac{|\log \mu_s|}{\log 2}\right\}.
\]

If $G$ is compact, nilpotent or soluble than $\mu_s = \mu_u =1$, so the distortion assumption automatically holds.
To formulate a result for  connected Lie  groups $G$ we need the following.  
 
 \begin{definition}
  Let 
$G_f$ be the closed subgroup of $G$ generated by the image $f(\Sigma_A)$ and let 
$Z(G_f) = \{z \in H_f \hbox{ : }  zh = hz \hbox{ for all } h \in G_f\}$
 denote its centre.  
 \end{definition}

We have the following general result.
 
 \begin{theorem}\label{shifts-groups}
 Let $\sigma: \Sigma_A \to \Sigma_A$ be a transitive subshift of finite type, $\Gamma$ 
a countable group and $\psi : \Sigma_A \to \Gamma$ a function satisfying 
$\psi((x_n)_{n=-\infty}^\infty)=\psi(x_0)$,
such that $T_\psi : \Sigma_A \times \Gamma \to \Sigma_A \times \Gamma$
is transitive.
Let $G$ be a connected Lie group and let $f : \Sigma_A \to G$ be a H\"older continuous function satisfying
the distortion assumption, such that 
\[
f_n(x) =e_G \mbox{ whenever } \psi_n(x) =e_\Gamma.
\]
 Then 
$$
f(x) = u(x) u(\sigma x)^{-1} \alpha(\psi(x))^{-1},
$$
where $u: \Sigma \to G$ is H\"older continuous and $\alpha: \Gamma \to Z(G_f)$ is a 
homomorphism.
\end{theorem}

The proof is a fairly routine generalization (see \cite{PW} or \cite{LW}), except for the fact that $\alpha$ is 
valued in $Z(G_f)$. To see this, we first observe that the proof gives us a function
$\widehat u : \Sigma_A \to G_f$ satisfying 
\[
\widehat u(T_\psi(x,\gamma))\widehat u(x,\gamma)^{-1} = f(x)
\quad \mbox{and} \quad
\widehat u (x,\eta\gamma) = \alpha(\gamma) \widehat u(x,\eta),
\]
for all $x \in \Sigma_A$ and all $\gamma, \eta \in \Gamma$.
Defining $u : \Sigma_A \to G_f$ by $u(x) = \widehat u(x,e_\Gamma)$, we calculate that,
for all $\gamma \in \Gamma$,
\begin{align*}
f(x) &= \widehat u(T_\psi(x,\gamma))\widehat u(x,\gamma)^{-1} \\
&= \widehat u(\sigma x,\psi(x)\gamma) \widehat u(x,\gamma)^{-1}
\\
&= \alpha(\gamma) \widehat u(\sigma x,\psi(x)) \widehat u(x,e_\Gamma)^{-1}\alpha(\gamma)^{-1}
\\
&= \alpha(\gamma) \widehat u(T_\psi(x,e_\Gamma)) \widehat u(x,e_\Gamma)^{-1}\alpha(\gamma)^{-1}
\\
&= \alpha(\gamma) f(x) \alpha(\gamma)^{-1}.
\end{align*}
We therefore conclude that $\alpha(\gamma) \in Z(G_f)$, as required.

\begin{remark}
Since $\alpha$ is valued in the abelian group $Z(G_f)$, it has the form $\alpha = \bar \alpha \circ \phi$, where
$\phi : \Gamma \to \Gamma^{\mathrm{ab}}$ is the abelianization homomorphism and $\bar \alpha: \Gamma^{\mathrm{ab}} 
\to Z(G_f)$ is a homomorphism.
\end{remark}

\begin{question}
Is the distortion assumption a necessary hypothesis for Theorem \ref{shifts-groups}?
In Kalinin's Liv\v{s}ic theorem (for matrix valued cocycles) a distortion result was not required \cite{Ka}.   However, it is not clear how this approach could be adapted to the present setting.  
\end{question}

\end{document}